\documentclass[11pt,reqno]{amsart}

\usepackage[T1]{fontenc}
\usepackage[utf8]{inputenc}
\usepackage{mathpazo}
\usepackage[scaled=.95]{helvet}
\usepackage{microtype}
\usepackage[margin=1.15in]{geometry}
\usepackage{amsmath,amssymb,mathtools}
\usepackage{xcolor}
\usepackage{hyperref}
\usepackage[nameinlink,noabbrev]{cleveref}

\hypersetup{
  colorlinks=true,
  linkcolor=blue!60!black,
  citecolor=blue!60!black,
  urlcolor=blue!60!black,
  pdfauthor={Sobral-Teixeira},
  pdftitle={Sobral-Teixeira}
}

\numberwithin{equation}{section}

\usepackage{aliascnt}

\theoremstyle{plain}
\newtheorem{theorem}{Theorem}[section]

\newaliascnt{proposition}{theorem}
\newtheorem{proposition}[proposition]{Proposition}
\aliascntresetthe{proposition}

\newaliascnt{lemma}{theorem}
\newtheorem{lemma}[lemma]{Lemma}
\aliascntresetthe{lemma}

\newaliascnt{corollary}{theorem}
\newtheorem{corollary}[corollary]{Corollary}
\aliascntresetthe{corollary}

\newaliascnt{conjecture}{theorem}

\aliascntresetthe{conjecture}

\theoremstyle{definition}

\newaliascnt{definition}{theorem}
\newtheorem{definition}[definition]{Definition}
\aliascntresetthe{definition}

\newaliascnt{assumption}{theorem}

\aliascntresetthe{assumption}

\theoremstyle{remark}

\newaliascnt{remark}{theorem}
\newtheorem{remark}[remark]{Remark}
\aliascntresetthe{remark}

\newcommand{\R}{\mathbb R}
\newcommand{\Rn}{\mathbb R^n}
\newcommand{\Sn}{\mathbb S^n}
\newcommand{\norm}[1]{\left\lVert #1\right\rVert}
\newcommand{\abs}[1]{\left\lvert #1\right\rvert}
\newcommand{\Acal}{\mathcal A}

\newcommand{\Adist}[2]{\operatorname{dist}_{#1}\!\left(#2,\Acal\right)}
\newcommand{\dist}{\mathrm{dist}}
\newcommand{\diam}{\mathrm{diam}}

\usepackage{parskip} 
\title[Endpoint Differentiability]{Endpoint Differentiability Moduli for Fully Nonlinear Elliptic Equations}

\author[A. Sobral]{Aelson Sobral}
\address{Applied Mathematics and Computational Sciences (AMCS), Computer, Electrical and Mathematical Sciences and Engineering Division (CEMSE), King Abdullah University of Science and Technology (KAUST), Thuwal, 23955-6900, Kingdom of Saudi Arabia}
\email{aelson.sobral@kaust.edu.sa}

\author[E. V. Teixeira]{Eduardo V. Teixeira}
\address{Department of Mathematics, Oklahoma State University, 74078, Stillwater, OK USA}{}
\email{eduardo.teixeira@okstate.edu}


\begin{document}

\begin{abstract}
A classical consequence of Caffarelli's fully nonlinear regularity theory
\textup{[L. A. Caffarelli, Ann. of Math. (2) \textbf{130} (1989),
no. 1, 189--213]} is that 
viscosity solutions of uniformly elliptic equations
\(F(D^2u)=f\), with \(f\in L^p\), \(p>n\), are locally $C^{1,\gamma}$ for every
$\gamma<\min\{\alpha_H,\sigma_p\}$.  Here
$\alpha_H=\alpha_H(n,\lambda,\Lambda) \in (0,1)$ denotes the  universal H\"older continuity exponent for gradient regularity of the homogeneous PDE, $F(D^2h) = 0$, and
$\sigma_p:=1-\frac{n}{p}$
is the scaling exponent of the source term.  In the source-limited regime,
$\sigma_p<\alpha_H$, the singularity of $f$ is the decisive obstruction, and the
corresponding endpoint regularity, $\gamma=\sigma_p$, is known to be attainable.  The
homogeneous-limited regime $\alpha_H\le\sigma_p$ is considerably more subtle and the
classical theory only reaches the exponents $\gamma<\alpha_H$.  Thus the
limiting differentiability estimate dictated by the homogeneous equation
remains unquantified.  This is the central endpoint gap
addressed in the present paper. In the strict homogeneous-limited regime, $\alpha_H<\sigma_p$, we prove that solutions admit
pointwise Taylor expansions with remainder
$$
        |u(x)-u(x_0)-Du(x_0)\cdot(x-x_0)|
        \lesssim
        |x-x_0|^{1+\alpha_H}
        \left(1+\log\frac1{|x-x_0|}\right)^m .
$$
Thus the homogeneous differentiability scale is reached up to an explicit
logarithmic defect.  At the critical threshold, $\alpha_H=\sigma_p$, finite
logarithmic powers no longer close the iteration; nevertheless, a slower
selection of scales yields the sub-power endpoint modulus
$$
        |u(x)-u(x_0)-Du(x_0)\cdot(x-x_0)|
        =
        O\left(
        |x-x_0|^{1+\alpha_H}
        \exp\left(
        A\sqrt{1+\log\frac1{|x-x_0|}}
        \right)
        \right).
$$
Both estimates improve the full family of classical sub-endpoint
$C^{1,\gamma}$ bounds, $\gamma<\alpha_H$, by providing quantitative
differentiability information at the limiting homogeneous exponent. The proof identifies a new scale-selection mechanism for endpoint Campanato-type recurrences, suggesting a flexible tool for regularity problems
where the limiting smoothness is dictated by the homogeneous theory itself.  We also
discuss the role and possible optimality of the logarithmic defects that arise at this endpoint scale.
\tableofcontents

\end{abstract}

\maketitle

\section{Introduction}

Interior $C^{1,\alpha}$ estimates for fully nonlinear elliptic equations stand
among the landmark achievements of modern nonlinear PDE theory.  In the
foundational works of Caffarelli \cite{Caffarelli1989} and Trudinger
\cite{Trudinger1988}, these estimates emerged as a central mechanism by which
uniform ellipticity yields compactness, stability, and ultimately differentiability of
solutions.

The general uniformly elliptic theory, however, has an intrinsic limitation.
Without convexity, concavity, or additional structure, one cannot in general
go beyond the universal $C^{1,\alpha}$ threshold. This limitation is not merely a byproduct of the available techniques; it is
an intrinsic feature of the nonlinear diffusion encoded by general uniformly
elliptic operators.  One of the central achievements of the
Nadirashvili--Vl\u{a}du\c{t} program \cite{NadirashviliVladut2007,NadirashviliVladut2008,
NadirashviliVladut2011,NadirashviliVladut2013a,
NadirashviliVladut2013b} was to make this obstruction explicit,
through the construction of singular and nonclassical solutions.   

Thus, for a general uniformly elliptic operator $F$, the homogeneous equation
\[
        F(D^2h)=0
\]
comes with a distinguished exponent
\[
        \alpha_H=\alpha_H(n,\lambda,\Lambda)\in(0,1),
\]
for which solutions are universally $C^{1,\alpha_H}$, while no better exponent
is available under uniform ellipticity alone.

Endpoint phenomena at limiting scales arise in several forms in the
regularity theory of fully nonlinear equations.  One line of work concerns
critical integrability of the right-hand side and the resulting Sobolev,
Campanato--John--Nirenberg, or borderline estimates
\cite{CaffarelliHuang2003,Escauriaza1993,
DaskalopoulosKuusiMingione2014,Swiech1997,Teixeira2014}.  In that setting,
the limiting behavior is governed primarily by the size, singularity, or
scaling of the inhomogeneity.  A complementary line of work obtains sharper
regularity from additional structure in the equation, through geometric,
degenerate, asymptotic, or potential-theoretic mechanisms
\cite{TeixeiraUrbano2014,AraujoRicarteTeixeira2015,
PimentelTeixeira2016,Teixeira2016,PimentelSantosTeixeira2022,
PimentelWalker2023,AraujoSobralTeixeira2025}.

The endpoint considered in the present paper is of a different kind.  It is
not a peripheral borderline refinement attached to the datum, but a question
at the core of the general uniformly elliptic theory.  It asks what remains of
Caffarelli's perturbative regularization mechanism when the source term is no
longer the limiting obstruction and the decisive scale is instead imposed by
the homogeneous nonlinear diffusion itself.  We
consider viscosity solutions of
\begin{equation}\label{eq:intro-equation}
        F(D^2u)=f
        \qquad\text{in }\Omega\subset\Rn ,
\end{equation}
where $F\colon\Sn\to\mathbb R$ is $(\lambda,\Lambda)$-elliptic, $F(0)=0$, and
$f\in L^p(\Omega)$, $p>n$.  The source term has the natural scaling exponent
\[
        \sigma_p \coloneqq 1-\frac{n}{p}.
\]
Caffarelli's perturbative theory \cite{Caffarelli1989} gives the sharp
open-range estimate
\[
        u\in C^{1,\gamma}_{\mathrm{loc}}(\Omega)
        \qquad
        \text{for every }
        0<\gamma<\min\{\alpha_H,\sigma_p\}.
\]
This estimate reflects the competition between two distinct scales: the
inhomogeneous scale $\sigma_p$, dictated by the right-hand side, and the
homogeneous scale $\alpha_H$, dictated by the universal $C^{1,\alpha_H}$
regularity of solutions to $F(D^2h)=0$. When $\sigma_p<\alpha_H$, the source term is the limiting feature, and the
endpoint dictated by scaling can be reached.  The homogeneous-limited regime
\[
        \alpha_H\le \sigma_p
\]
has a different character.  In this case the datum is, from the viewpoint of
scaling, at least as regular as the homogeneous theory permits; the limiting
barrier is therefore not the roughness of $f$, but the universal
differentiability ceiling imposed by the homogeneous fully nonlinear equation
itself.  Thus the question lies at the core of the general theory: can the classical
open-range estimate be promoted to a quantitative modulus at the intrinsic
homogeneous scale?

The classical argument stops short of this limit.  It yields estimates at
every exponent $\gamma<\alpha_H$, but it does not produce a modulus at
$\gamma=\alpha_H$.  In this sense, the universal homogeneous exponent marks a genuine endpoint
gap in the quantitative theory.  The purpose of this paper is to resolve this
gap at the level of quantitative moduli, reaching the homogeneous
differentiability scale up to explicit logarithmic corrections.

In the strict homogeneous-limited regime $\alpha_H<\sigma_p$, we prove that solutions admit
pointwise Taylor expansions at the universal homogeneous differentiability
scale.  More precisely, for every
$x_0\in\Omega$ and every $B_R(x_0)\Subset\Omega$, we show 
\begin{equation}\label{eq:intro-log-estimate}
        \sup_{B_s(x_0)} |u-\left [ u(x_0)+Du(x_0)\cdot(x-x_0) \right ]|
        \le
        C\Theta_{R,-}(x_0)\,
        s^{1+\alpha_H}
        \left(1+\log\frac Rs\right)^m,
        \qquad 0<s\le R,
\end{equation}
where
\[
        \Theta_{R,-}(x_0)
        \coloneqq
        \frac{\Adist{B_R(x_0)}{u}}{R^{1+\alpha_H}}
        +
        R^{1-\frac np-\alpha_H}\|f\|_{L^p(B_R(x_0))}.
\]
Since the logarithmic factor is lower order than any loss of power,
\eqref{eq:intro-log-estimate} is stronger than the full family of classical
sub-endpoint estimates $C^{1,\gamma}$, $\gamma<\alpha_H$.  Its main significance,
however, is that it supplies a universal modulus of differentiability at the
homogeneous scaling level itself.  Thus the limiting exponent dictated by the
fully nonlinear homogeneous theory is no longer only approached from below; it
is reached quantitatively, with the residual endpoint obstruction measured by
an explicit logarithmic correction.

The form of the endpoint modulus depends on the distance between the source
scale and the homogeneous scale.  In the strict homogeneous-limited regime
$\alpha_H<\sigma_p$, the inhomogeneous term carries a positive excess of
scaling.  This excess can be spent to compensate for the noncontractive
homogeneous constant, producing only a finite logarithmic correction.  At the
critical threshold $\alpha_H=\sigma_p$, this excess disappears.  The same
fixed logarithmic normalization no longer controls the perturbative
accumulation, and the scale selection must be slowed down, leading instead to
a sub-power correction.

For $A>0$ sufficiently large,
\begin{equation}\label{eq:intro-critical-estimate}
        \sup_{B_s(x_0)} |u-P_{x_0}|
        \le
        C\Theta_R(x_0)\,
        s^{1+\alpha_H}
        \exp\left(
        A\sqrt{1+\log\frac Rs}
        \right),
        \qquad 0<s\le R,
\end{equation}
where
\[
        \Theta_R(x_0)
        \coloneqq
        \frac{\Adist{B_R(x_0)}{u}}{R^{1+\alpha_H}}
        +
        \|f\|_{L^p(B_R(x_0))}.
\]
The correction in \eqref{eq:intro-critical-estimate} grows faster than any
fixed power of $\log(R/s)$, but slower than every negative power of $s$.
Consequently, the critical estimate also improves every
$C^{1,\gamma}$ bound with $\gamma<\alpha_H$.

These pointwise estimates yield, in particular, interior moduli of continuity
for the gradient.  For every $\Omega'\Subset\Omega$, the solution is
differentiable in $\Omega'$, and $Du$ satisfies the corresponding logarithmic
H\"older-type modulus.  Thus the classical sub-endpoint differentiability
theory is upgraded to quantitative endpoint information at the homogeneous
scale.

We now describe the obstruction and the mechanism behind the proof.  The
homogeneous theory gives the affine approximation estimate
\begin{equation}\label{eq:intro-homogeneous-estimate}
        \Adist{B_\rho}{h}
        \le
        C_H\rho^{1+\alpha_H}\|h\|_{L^\infty(B_1)},
        \qquad 0<\rho<1,
\end{equation}
where $\Adist{B_\rho}{h}$ denotes the $L^\infty(B_\rho)$-distance from $h$ to
the space of affine functions.  Combining \eqref{eq:intro-homogeneous-estimate}
with ABP comparison and a homogeneous replacement argument yields the two-scale
excess inequality
\begin{equation}\label{eq:intro-two-scale-excess}
        \Adist{B_s(x_0)}{u}
        \le
        C_H\left(\frac{s}{r}\right)^{1+\alpha_H}
        \Adist{B_r(x_0)}{u}
        +
        C r^{2-\frac np}\|f\|_{L^p(B_r(x_0))}
\end{equation}
for $0<s<r$. This recurrence contains the endpoint obstruction in its simplest form.  If
one normalizes by the limiting scale $r^{1+\alpha_H}$, the homogeneous term
has exactly the correct decay, but it carries the noncontractive constant
$C_H$.  A fixed-scale Campanato iteration avoids this obstruction by lowering
the exponent; this is precisely what yields the classical family
$C^{1,\gamma}$, $\gamma<\alpha_H$.

The point of the present argument is to keep the homogeneous exponent and
alter the scales instead.  The radii are chosen dynamically so that the endpoint
normalization absorbs the noncontractive homogeneous constant, while the
inhomogeneous term remains summable at the same decay level.  Thus the loss is
not paid in the exponent, but in an explicit endpoint modulus. 

A useful feature of the argument is its stability under quantitative
improvements of the homogeneous theory.  The constants $\alpha_H$ and $C_H$
enter only through the homogeneous affine approximation estimate.  Thus,
whenever a subclass of uniformly elliptic operators admits a sharper or more
explicit homogeneous $C^{1,\alpha}$ estimate, the same scale-selection
argument immediately produces the corresponding endpoint modulus for the
inhomogeneous problem.  In this sense, the theorem converts quantitative
information at the homogeneous level into quantitative differentiability at
the endpoint scale.

A second conceptual outcome of the proof is an abstract endpoint Campanato
principle.  The relevant analytic input is a two-scale recurrence of the form
\[
        E(s)
        \le
        C_0\left(\frac{s}{r}\right)^\beta E(r)
        +
        \Phi(r),
        \qquad 0<s<r,
\]
where $\beta$ is the limiting decay exponent supplied by the homogeneous
theory.  Such a recurrence is not, by itself, an endpoint estimate: the
constant $C_0$ generally prevents contraction at exponent $\beta$ on fixed
geometric scales.  The scale-selection mechanism developed here converts this
noncontractive recurrence into a quantitative endpoint modulus, without
sacrificing the homogeneous exponent.

Recurrences of this type arise naturally in boundary regularity, anisotropic
parabolic problems, nonlocal equations with tail terms, critical-point degenerate models, and higher-order
approximation schemes, among other settings.  Thus the scale-selection
algorithm isolated here may provide a flexible tool for endpoint regularity
problems beyond the inhomogeneous fully nonlinear elliptic model treated in
this paper.

Finally, we discuss the role of logarithmic defects and their possible
optimality.  Classical endpoint examples show that logarithmic losses are
natural, and in some regimes sharp, at critical scales.  We also examine
singular-profile tests related to the Nadirashvili--Vl\u{a}du\c{t} program.
These examples show that logarithmic defects are visible at the critical scale
$\sigma_p=\alpha_H$, while also explaining why standard singular profiles do
not decide the sharpness of the logarithmic loss in the strict
homogeneous-limited regime $\alpha_H<\sigma_p$.

The paper is organized as follows.  In \Cref{sec:setting} we recall the
ellipticity assumptions, the $L^p$-viscosity framework, and the homogeneous
$C^{1,\alpha_H}$ estimate.  In \Cref{sec:replacement} we prove the homogeneous
replacement estimate and derive the two-scale excess inequality
\eqref{eq:intro-two-scale-excess}.  In \Cref{sec:iteration} we prove the
logarithmic endpoint estimate in the strict homogeneous-limited regime
$\alpha_H<\sigma_p$.  In \Cref{sec:borderline-equality} we treat the critical
case $\alpha_H=\sigma_p$.  In \Cref{sec:abstract} we extract the common
iteration as an abstract scale-selection principle for endpoint Campanato
recurrences. Finally, in \Cref{sec:optimality} we discuss the role and
possible optimality of logarithmic defects in endpoint regularity theory.

\section{Setting and homogeneous input}\label{sec:setting}

Throughout, \(n\ge2\), and \(\Sn\) denotes the space of real symmetric
\(n\times n\) matrices. Let \(0<\lambda\le\Lambda\).  The Pucci extremal
operators are
\[
        \mathcal M^+_{\lambda,\Lambda}(X)
        \coloneqq
        \Lambda\sum_{e_i>0}e_i+
        \lambda\sum_{e_i<0}e_i,
        \qquad
        \mathcal M^-_{\lambda,\Lambda}(X)
        \coloneqq
        \lambda\sum_{e_i>0}e_i+
        \Lambda\sum_{e_i<0}e_i,
\]
where \(e_i\) are the eigenvalues of \(X\in\Sn\).

\begin{definition}
A continuous operator \(G\colon \Sn\to\R\), is \((\lambda,\Lambda)\)-elliptic if
\begin{equation}\label{eq:ellipticity}
        \mathcal M^-_{\lambda,\Lambda}(X-Y)
        \le
        G(X)-G(Y)
        \le
        \mathcal M^+_{\lambda,\Lambda}(X-Y)
\end{equation}
for every \(X,Y\in\Sn\).
\end{definition}

Throughout the paper, we denote by $\mathfrak{F}_{\lambda,\Lambda}$ the family of continuous \((\lambda,\Lambda)\)-elliptic operators \(G\colon\Sn\to\R\), normalized by \(G(0)=0\).

All solutions with a nonzero right-hand side are understood in the \(L^p\)-viscosity sense. This convention is harmless for the bounded right-hand sides and is the natural framework for the stability and comparison statements used below. We refer to \cite{CrandallIshiiLions1992, CCKS1996}
for the viscosity framework and to the latter work for equations with
measurable ingredients.

\begin{theorem}\label{thm:homogeneous-estimate}
There exist constants
\[
        \alpha_H=\alpha_H(n,\lambda,\Lambda)\in(0,1),
        \qquad
        C_H=C_H(n,\lambda,\Lambda)\ge 1,
\]
such that the following holds. For every \(G\in\mathfrak{F}_{\lambda,\Lambda}\) and every bounded viscosity solution \(h\in C(B_1)\) of
\[
        G(D^2h)=0
        \qquad\text{in }B_1,
\]
one has \(h\in C^{1,\alpha_H}_{\mathrm{loc}}(B_1)\). Moreover,
\begin{equation}\label{eq:homogeneous-estimate}
        \Adist{B_r}{h}
        \le
        C_H r^{1+\alpha_H}
        \norm{h}_{L^\infty(B_1)}
        \qquad\text{for every }0<r<1 .
\end{equation}
Here \(\Acal\) denotes the space of affine functions on \(\mathbb{R}^n\).
\end{theorem}

This is the standard interior \(C^{1,\alpha}\) estimate for viscosity solutions of homogeneous uniformly elliptic fully nonlinear equations. See, for instance, Caffarelli--Cabré~\cite[Section 5.3]{CaffarelliCabre1995}.

Finally, given a number \(R>0\) we define
\begin{equation}\label{eq:log-weight}
        L_R(r)
        \coloneqq
        1+\log\frac{R}{r},
        \qquad 0<r\le R.
\end{equation}

\section{Comparison with the homogeneous replacement}\label{sec:replacement}

The first ingredient is a comparison with the homogeneous replacement. We state it in the form needed for $L^p$, with $p \geq n$, right-hand sides.

\begin{lemma}\label{lem:homogeneous-replacement}
Consider \(G\in\mathfrak{F}_{\lambda,\Lambda}\), \(g\in L^p(B_1)\) for $p \geq n$, and let \(w,h\in C(\overline B_1)\) satisfy
\[
        G(D^2w)=g, \qquad G(D^2h)=0 \qquad \text{in } B_1,
\]
with \(w=h\) on \(\partial B_1\). Then
\begin{equation}\label{eq:homogeneous-replacement}
        \norm{w-h}_{L^\infty(B_1)}
        \le
        C\norm{g}_{L^p(B_1)},
\end{equation}
where \(C\) depends only on \(n,p,\lambda,\Lambda\).
\end{lemma}

\begin{proof}
Set \(z=w-h\). Both \(z\) and \(-z\) satisfy
\[
        \mathcal M^+_{\lambda,\Lambda}(D^2v)\ge -|g|
        \qquad\text{in }B_1
\]
in the viscosity sense. Since \(z=0\) on \(\partial B_1\), applying ABP to \(z\) and \(-z\) gives
\[
        \norm{z}_{L^\infty(B_1)}
        \le C\norm{g}_{L^n(B_1)}
        \le C\norm{g}_{L^p(B_1)} .
\]
Thus \(\norm{w-h}_{L^\infty(B_1)}\le C\norm{g}_{L^p(B_1)}\).
\end{proof}

Accordingly, we call \(h\) the \(G\)-homogeneous replacement of \(w\) in \(B_1\).

\begin{lemma}\label{lem:two-scale}
Let \(F\in\mathfrak{F}_{\lambda,\Lambda}\), and let \(u\in C(\overline{B} _R(x_0))\) be an \(L^p\)-viscosity solution of
\[
        F(D^2u)=f \qquad\text{in }B_R(x_0),
\]
with \(f\in L^p(B_R(x_0))\). Then, for every \(0<s<r\le R\),
\begin{equation}\label{eq:two-scale-excess}
        \Adist{B_s(x_0)}{u}
        \le
        C_H\left(\frac{s}{r}\right)^{1+\alpha_H}
        \Adist{B_r(x_0)}{u}
        +
        C r^{2 - \frac{n}{p}}\norm{f}_{L^p(B_r(x_0))},
\end{equation}
where \(C\) depends only on \(n,p,\lambda,\Lambda,C_H\).
\end{lemma}

\begin{proof}
Fix \(0<s<r\le R\) and let \(\eta>0\) be arbitrary. Choose \(\ell\in\Acal\) such that
\[
        \norm{u-\ell}_{L^\infty(B_r(x_0))} \le \Adist{B_r(x_0)}{u}+\eta .
\]
For $y \in B_1$, we define $w(y) \coloneqq u(x_0+ry)-\ell(x_0+ry)$, which solves
\[
        G_r(D^2w)=g_r(y)
        \qquad\text{in }B_1,
\]
where
\begin{equation}\label{eq:scaled-G-g}
        G_r(M)
        \coloneqq
        r^2F(r^{-2}M),
        \qquad
        g_r(y)
        \coloneqq
        r^2f(x_0+ry).
\end{equation}
Let \(h\) be the $G_r$-homogeneous replacement of \(w\) in $B_1$. By \Cref{lem:homogeneous-replacement},
\begin{equation}\label{eq:w-h-two-scale}
        \norm{w-h}_{L^\infty(B_1)}
        \le
        C r^{2-\frac{n}{p}}\norm{f}_{L^p(B_r(x_0))} .
\end{equation}
Set \(\theta=s/r\). By \Cref{thm:homogeneous-estimate}, there is an affine function \(L\in\Acal\) such that
\[
        \norm{h-L}_{L^\infty(B_\theta)}
        \le
        C_H\theta^{1+\alpha_H}\norm{h}_{L^\infty(B_1)} + \eta.
\]
Since \(\norm{h}_{L^\infty(B_1)} \le \norm{w}_{L^\infty(B_1)} + \norm{w-h}_{L^\infty(B_1)}\), we obtain, using \(\theta^{1+\alpha_H}\le1\),
\[
\begin{aligned}
        \norm{w-L}_{L^\infty(B_\theta)}
        &\le
        C_H\theta^{1+\alpha_H}\norm{w}_{L^\infty(B_1)} + \eta
        +
        C r^{2-\frac{n}{p}}\norm{f}_{L^p(B_r(x_0))}\\
        &\le
        C_H\theta^{1+\alpha_H}\bigl(\Adist{B_r(x_0)}{u}+\eta\bigr)
        + \eta + 
        C r^{2-\frac{n}{p}}\norm{f}_{L^p(B_r(x_0))}.
\end{aligned}
\]
Rescaling back to $u$ and letting \(\eta\downarrow0\) proves \eqref{eq:two-scale-excess}.
\end{proof}

\section{Campanato-type iteration on logarithmic scales}\label{sec:iteration}

In this section, we prove the regularity estimate with the logarithmic loss for which the exponent $\alpha_H$, coming from the regularity of the homogeneous equation, satisfies
\[
    \alpha_H < 1 - \frac{n}{p},
\]
meaning that the forcing term is of higher order than the homogeneous \(C^{1,\alpha_H}\) scale.

\begin{theorem}\label{thm:pointwise-log}
Let
\begin{equation}\label{eq:Q-choice}
        1 < Q \leq \frac{2 - \frac{n}{p}}{1+\alpha_H},
\end{equation}
and choose a positive number \(m\) such that
\begin{equation}\label{eq:m-choice}
        \mu\coloneqq C_H Q^{-m}<1 .
\end{equation}
Let \(F\in\mathfrak{F}_{\lambda,\Lambda}\), and let \(u\in C(\overline{B_R(x_0)})\) be an \(L^p\)-viscosity solution of
\[
        F(D^2u)=f
        \qquad\text{in }B_R(x_0),
\]
with \(f\in L^p(B_R(x_0))\). Then, $u$ is differentiable at $x_0$, and 
\begin{equation}\label{eq:pointwise-log-taylor}
        \abs{u(x)-u(x_0)-Du(x_0)\cdot(x-x_0)}
        \le
        C\Theta_R(x_0)
        \abs{x-x_0}^{1+\alpha_H}
        \left(1+\log\frac{R}{\abs{x-x_0}}\right)^m
\end{equation}
for every \(x\in B_R(x_0)\). The constant \(C\) depends only on
\(n,p,\lambda,\Lambda,C_H,\alpha_H,Q,m\), where
\[
        \Theta_R(x_0)
        :=
        \frac{\dist_{B_R(x_0)}(u,\mathcal A)}{R^{1+\alpha_H}}
        +
        R^{1-\frac np-\alpha_H}
        \|f\|_{L^p(B_R(x_0))}.
\]
\end{theorem}

\begin{proof}
We show that there exists an affine function
\[
        P_{x_0}(x)=u(x_0)+p_{x_0}\cdot(x-x_0)
\]
such that, for every \(0<s\le R\),
\begin{equation}\label{eq:pointwise-log-estimate}
        \sup_{B_s(x_0)}\abs{u-P_{x_0}}
        \le
        C\Theta_R(x_0) s^{1+\alpha_H} L_R(s)^m.
\end{equation}
We first prove the excess decay
\begin{equation}\label{eq:excess-log-bound-claim}
        \Adist{B_s(x_0)}{u}
        \le
        C\Theta_R(x_0)s^{1+\alpha_H}L_R(s)^m
        \qquad\text{for every }0<s\le R .
\end{equation}
Let
\[
        F_0\coloneqq \norm{f}_{L^p(B_R(x_0))},
        \qquad
        L(s)\coloneqq L_R(s)=1+\log\frac{R}{s}.
\]
The logarithmic scale is chosen to compensate for the noncontractive
homogeneous constant without destroying the higher-order decay of the forcing
term.  If \(L_R(r_{k+1})=Q L_R(r_k)\), then the normalization
\(r^{1+\alpha_H}L_R(r)^m\) produces the factor \(Q^{-m}\) in the homogeneous
part, while the admissibility condition
\[
        Q\le \frac{2-\frac np}{1+\alpha_H}
\]
keeps the perturbative term bounded under the same normalization.  Accordingly,
we define \((r_k)_{k\ge0}\) by
\[
        L_R(r_k)=Q^k,
        \qquad\text{equivalently}\qquad
        r_k=R\exp(1-Q^k).
\]
Thus \(r_0=R\), \(L(r_{k+1})=QL(r_k)\), and since $Q>1$, we have \(r_k\downarrow0\). Set
\[
        E_k\coloneqq \Adist{B_{r_k}(x_0)}{u},
        \qquad
        A_k\coloneqq \frac{E_k}{r_k^{1 + \alpha_H}L(r_k)^m}.
\]
Applying \Cref{lem:two-scale} with \(s=r_{k+1}\) and \(r=r_k\), we get
\[
        E_{k+1}
        \le
        C_H\left(\frac{r_{k+1}}{r_k}\right)^{1+\alpha_H}E_k
        +
        C r_k^{2-\frac{n}{p}}F_0.
\]
Dividing by \(r_{k+1}^{1+\alpha_H}L(r_{k+1})^m\) gives
\[
        A_{k+1}
        \le
        C_H\left(\frac{L(r_k)}{L(r_{k+1})}\right)^m A_k
        +
        C F_0\frac{r_k^{2-\frac{n}{p}}}{r_{k+1}^{1+\alpha_H}L(r_{k+1})^m}.
\]
The first coefficient is \(C_HQ^{-m}=\mu<1\). For the forcing term, write \(L_k=L(r_k)=Q^k\). Since
\[
        r_k=Re^{1-L_k},
        \qquad
        r_{k+1}=Re^{1-QL_k},
\]
we have
\[
\begin{aligned}
        \frac{r_k^{2-\frac{n}{p}}}{r_{k+1}^{1+\alpha_H}L(r_{k+1})^m}
        &=
        R^{1-\frac{n}{p}-\alpha_H}
        \frac{\exp\left\{\left(2-\frac{n}{p}\right)(1-L_k)-(1+\alpha_H)(1-QL_k)\right\}}{(QL_k)^m} \\
        &=
        R^{1-\frac{n}{p}-\alpha_H}
        \frac{\exp\left\{1-\frac{n}{p}-\alpha_H-\left(2-\frac{n}{p}-(1+\alpha_H)Q\right)L_k\right\}}{(QL_k)^m}.
\end{aligned}
\]
Because 
\[
	2-\frac{n}{p}-(1+\alpha_H)Q \geq 0,
\]
this expression is bounded by \(CR^{1-\frac{n}{p}-\alpha_H}\), with \(C=C(\alpha_H,Q,m)\). Hence
\[
        A_{k+1}
        \le
        \mu A_k
        +
        CR^{1-\frac{n}{p}-\alpha_H}F_0.
\]
Since $\mu<1$, iteration gives
\[
	    A_k
        \le
        C\left(A_0+R^{1-\frac{n}{p}-\alpha_H}F_0\right)
        \le
        C\Theta_R(x_0)
        \qquad\text{for every }k\ge 0.
\]
Therefore
\begin{equation}\label{eq:Ek-bound-log}
        E_k
        \le
        C\Theta_R(x_0)r_k^{1+\alpha_H}L(r_k)^m.
\end{equation}
Now let \(0<s\le R\), and choose \(k\ge0\) such that \(r_{k+1}<s\le r_k\). By \Cref{lem:two-scale},
\[
        \Adist{B_s(x_0)}{u}
        \le
        C_H\left(\frac{s}{r_k}\right)^{1+\alpha_H}E_k
        +
        C r_k^{2-\frac{n}{p}}F_0.
\]
Using \eqref{eq:Ek-bound-log} and \(L(s)\ge L(r_k)\), the first term is at most \(C\Theta_R(x_0)s^{1+\alpha_H}L(s)^m\). For the forcing term, since \(s>r_{k+1}\) and \(L(s)\ge1\),
\[
        \frac{r_k^{2-\frac{n}{p}}}{s^{1+\alpha_H}L(s)^m}
        \le
        \frac{r_k^{2-\frac{n}{p}}}{r_{k+1}^{1+\alpha_H}}
        =
        R^{1-\frac{n}{p}-\alpha_H}\exp\left\{1-\frac{n}{p}-\alpha_H-\left(2-\frac{n}{p}-(1+\alpha_H)Q\right)L_k\right\}
        \le
        CR^{1-\frac{n}{p}-\alpha_H}.
\]
Thus, \eqref{eq:excess-log-bound-claim} is proved since
\[
        Cr_k^{2-\frac{n}{p}}F_0
        \le
        C\Theta_R(x_0)s^{1+\alpha_H}L(s)^m.
\]
To obtain \eqref{eq:pointwise-log-estimate}, let
\[
        d_j\coloneqq 2^{-j}R, \qquad j=0,1,2,\dots.
\]
For each \(j\), choose an affine function
\[
        \ell_j(x) = c_j+p_j\cdot(x-x_0)
\]
such that
\begin{equation}\label{eq:ellj-choice}
        \norm{u-\ell_j}_{L^\infty(B_{d_j}(x_0))}
        \le
        2C\Theta_R(x_0)d_j^{1+\alpha_H}L_R(d_j)^m.
\end{equation}
Comparing \(\ell_{j+1}\) and \(\ell_j\) on \(B_{d_{j+1}}(x_0)\), and using \eqref{eq:ellj-choice} at levels \(j\) and \(j+1\), gives
\[
        \norm{\ell_{j+1}-\ell_j}_{L^\infty(B_{d_{j+1}}(x_0))}
        \le
        C\Theta_R(x_0)d_j^{1+\alpha_H}L_R(d_j)^m.
\]
For affine functions, the elementary estimate
\[
        \norm{q}_{L^\infty(B_\varrho(x_0))}\le\delta
        \quad\Longrightarrow\quad
        \abs{q(x_0)}\le\delta,
        \qquad
        \abs{Dq}\le C(n)\frac{\delta}{\varrho}
\]
gives
\begin{equation}\label{eq:affine-increments-log}
        \abs{c_{j+1}-c_j}
        \le
        C\Theta_R(x_0)d_j^{1+\alpha_H}L_R(d_j)^m,
\end{equation}
and
\begin{equation}\label{eq:slope-increments-log}
        \abs{p_{j+1}-p_j}
        \le
        C\Theta_R(x_0)d_j^{\alpha_H}L_R(d_j)^m.
\end{equation}
Since \(\alpha_H>0\), we obtain
\[
        \sum_{j=0}^\infty d_j^{\alpha_H}L_R(d_j)^m<\infty.
\]
Hence \((c_j)\) and \((p_j)\) are Cauchy sequences, and therefore convergent. Let
\[
        c_\infty=\lim_{j\to\infty}c_j,
        \qquad
        p_\infty=\lim_{j\to\infty}p_j.
\]
Since \(x_0\in B_{d_j}(x_0)\), \eqref{eq:ellj-choice} gives
\[
        \abs{u(x_0)-c_j}
        \le
        C\Theta_R(x_0)d_j^{1+\alpha_H}L_R(d_j)^m
        \longrightarrow 0.
\]
Thus \(c_\infty=u(x_0)\). Define
\[
        P_{x_0}(x)
        \coloneqq
        u(x_0) + p_\infty\cdot(x-x_0).
\]
The tails of \eqref{eq:affine-increments-log} and \eqref{eq:slope-increments-log} satisfy
\[
        \abs{c_j-u(x_0)}
        \le
        C\Theta_R(x_0)d_j^{1+\alpha_H}L_R(d_j)^m,
        \qquad
        \abs{p_j-p_\infty}
        \le
        C\Theta_R(x_0)d_j^{\alpha_H}L_R(d_j)^m.
\]
Consequently, for \(x\in B_{d_j}(x_0)\),
\[
\begin{aligned}
        \abs{\ell_j(x)-P_{x_0}(x)}
        &\le
        \abs{c_j-u(x_0)}+
        \abs{p_j-p_\infty}\,\abs{x-x_0} \\
        &\le
        C\Theta_R(x_0)d_j^{1+\alpha_H}L_R(d_j)^m.
\end{aligned}
\]
Combining this with \eqref{eq:ellj-choice}, we obtain
\[
        \sup_{B_{d_j}(x_0)}\abs{u-P_{x_0}}
        \le
        C\Theta_R(x_0)d_j^{1+\alpha_H}L_R(d_j)^m.
\]
For arbitrary \(0<s\le R\), choose \(j\) such that \(d_{j+1}<s\le d_j\). Since \(d_j\le2s\) and \(L_R(d_j)\le L_R(s)\), this gives \eqref{eq:pointwise-log-estimate}.

Finally, \(s^{1+\alpha_H}L_R(s)^m=o(s)\) as \(s\downarrow0\). Hence \eqref{eq:pointwise-log-estimate} implies differentiability at \(x_0\), and \(p_\infty=Du(x_0)\). The pointwise estimate \eqref{eq:pointwise-log-taylor} follows by taking \(s=\abs{x-x_0}\).
\end{proof}

\begin{remark}[Quantitative homogeneous input]
The proof of Theorem~\ref{thm:pointwise-log} uses the homogeneous theory only
through the affine approximation estimate
\[
        \mathrm{dist}_{B_\rho}(h,\mathcal A)
        \le
        C_H\rho^{1+\alpha_H}\|h\|_{L^\infty(B_1)}.
\]
Consequently, the same proof applies verbatim to any class of operators for
which one has a quantitative homogeneous estimate
\[
        \mathrm{dist}_{B_\rho}(h,\mathcal A)
        \le
        C_*\rho^{1+\alpha_*}\|h\|_{L^\infty(B_1)},
        \qquad 0<\rho<1,
\]
with explicit constants $\alpha_*\in(0,1)$ and $C_*\ge1$.  If
$\alpha_*<1-n/p$, then for every
\[
        1<Q\le \frac{2-\frac np}{1+\alpha_*}
\]
and every $m>0$ such that $C_*Q^{-m}<1$, the conclusion of
Theorem~\ref{thm:pointwise-log} holds with $\alpha_H$ and $C_H$ replaced by
$\alpha_*$ and $C_*$.  Thus any quantitative improvement in the homogeneous
regularity theory propagates directly to an endpoint differentiability modulus
for the inhomogeneous equation.
\end{remark}

As is often the case in Caffarelli's geometric approach to regularity theory,
\Cref{thm:pointwise-log} is fundamentally a pointwise expansion theorem: at a
fixed point, it quantifies the growth of the error after subtracting the
appropriate affine approximation.  In the present problem, this pointwise
estimate is not confined to a single base point; it holds uniformly throughout
compact subsets of the domain.  This uniformity is what converts the endpoint
growth estimate into a local continuity modulus for the gradient.

The underlying passage from uniform pointwise affine expansions to a two-point
estimate for the gradient is a standard device in the field, although it is
often left implicit.  We isolate it here as a general lemma, both for
completeness and to make clear that no additional PDE input is involved in
this final step.

\begin{definition}
A modulus of continuity is a continuous nondecreasing function
\[
        \omega:[0,R_*]\to[0,\infty)
\]
such that $\omega(0)=0$ and $\omega(r)>0$ for every $r>0$.  If
$K\subset\Omega$ and $v:K\to\mathbb R^n$, we write
\[
        [v]_{C^{0,\omega}(K)}
        :=
        \sup_{\substack{x,z\in K\\ x\neq z}}
        \frac{|v(x)-v(z)|}{\omega(|x-z|)}.
\]
We say that $u\in C^{1,\omega}_{\mathrm{loc}}(\Omega)$ if
$u$ is differentiable in $\Omega$ and, for every $K\Subset\Omega$,
\[
        [Du]_{C^{0,\omega}(K)}<\infty .
\]
\end{definition}

\begin{lemma}[Pointwise affine expansions imply a gradient modulus]
\label{lem:pointwise-to-gradient-modulus}
Let $\Omega\subset\mathbb R^n$ be open, let $\Omega'\Subset\Omega$ and let \(R_0>0\) be such that \(B_{R_0}(y)\subset\Omega\) for every
\(y\in\Omega'\). Let $\omega$ be a modulus of continuity on $[0,\diam(\Omega)]$.  Suppose that
$u\in L^\infty(\Omega)$ and that there exists $M\ge0$ with the following
property: for every $y\in\Omega'$ there is a vector $p_y\in\mathbb R^n$ such
that
\begin{equation}\label{eq:abstract-pointwise-expansion}
        \sup_{B_\rho(y)}
        |u(x)-u(y)-p_y\cdot(x-y)|
        \le
        M\rho\,\omega(\rho)
\end{equation}
for every $0<\rho\le R_0$.

Then $u$ is differentiable at every point of $\Omega'$, with
\[
        Du(y)=p_y,
        \qquad y\in\Omega',
\]
and, for every $x,z\in\Omega'$ with $0<|x-z|<R_0/4$,
\begin{equation}\label{eq:small-gradient-modulus}
        |Du(x)-Du(z)|
        \le
        C_n M\,\omega(|x-z|).
\end{equation}
Consequently, $Du\in C^{0,\omega}(\Omega')$ and
\begin{equation}\label{eq:global-gradient-modulus}
        [Du]_{C^{0,\omega}(\Omega')}
        \le
        C_n M
        +
        \frac{C_n}{\omega(R_0/4)}
        \left(
        \frac{\|u\|_{L^\infty(\Omega)}}{R_0}
        +
        M\omega(R_0)
        \right).
\end{equation}
In particular, if \eqref{eq:abstract-pointwise-expansion} holds locally in
$\Omega$ for every $\Omega'\Subset\Omega$, then
$u\in C^{1,\omega}_{\mathrm{loc}}(\Omega)$.
\end{lemma}

\begin{proof}
First, \eqref{eq:abstract-pointwise-expansion} implies differentiability at
$y$.  Indeed,
\[
        \frac{|u(x)-u(y)-p_y\cdot(x-y)|}{|x-y|}
        \le
        M\omega(|x-y|)
        \longrightarrow 0
        \qquad\text{as }x\to y.
\]
Thus $Du(y)=p_y$ for every $y\in\Omega'$.

We now prove the modulus estimate.  Let $x,z\in\Omega'$ and put
\[
        d:=|x-z|.
\]
Assume first that $0<d<R_0/4$, and set
\[
        \nu:=\frac{z-x}{|z-x|}.
\]
Applying \eqref{eq:abstract-pointwise-expansion} at $x$ and evaluating at $z$
gives
\[
        u(z)
        =
        u(x)+Du(x)\cdot(z-x)
        +
        O\bigl(Md\,\omega(d)\bigr).
\]
Similarly, applying \eqref{eq:abstract-pointwise-expansion} at $z$ and
evaluating at $x$ gives
\[
        u(x)
        =
        u(z)+Du(z)\cdot(x-z)
        +
        O\bigl(Md\,\omega(d)\bigr).
\]
Adding these two identities yields
\begin{equation}\label{eq:abstract-normal-component}
        |(Du(x)-Du(z))\cdot\nu|
        \le
        C M\omega(d).
\end{equation}

To estimate the remaining components, choose unit vectors
$\nu_2,\ldots,\nu_n$ so that
\[
        \{\nu,\nu_2,\ldots,\nu_n\}
\]
is an orthonormal basis of $\mathbb R^n$.  For each $i=2,\ldots,n$, define
\[
        x_i
        :=
        \frac{x+z}{2}
        +
        \frac{\sqrt3}{2}\,d\,\nu_i .
\]
Then
\[
        |x_i-x|=|x_i-z|=d.
\]
Since $d<R_0/4$, all the points involved lie in $\Omega$ and the estimates
\eqref{eq:abstract-pointwise-expansion} may be applied at both $x$ and $z$.

Applying \eqref{eq:abstract-pointwise-expansion} at $x$ and at $z$, both
evaluated at $x_i$, gives
\begin{align*}
        u(x_i)
        &=
        u(x)+Du(x)\cdot(x_i-x)
        +
        O\bigl(Md\,\omega(d)\bigr),                                      \\
        u(x_i)
        &=
        u(z)+Du(z)\cdot(x_i-z)
        +
        O\bigl(Md\,\omega(d)\bigr).
\end{align*}
Subtracting the two identities, we obtain
\begin{align}
        0
        &=
        u(x)-u(z)
        +
        Du(x)\cdot(x_i-x)
        -
        Du(z)\cdot(x_i-z)
        +
        O\bigl(Md\,\omega(d)\bigr).
        \label{eq:abstract-tangential-raw}
\end{align}
Since
\[
        x_i-x=\frac d2\nu+\frac{\sqrt3 d}{2}\nu_i,
        \qquad
        x_i-z=-\frac d2\nu+\frac{\sqrt3 d}{2}\nu_i,
\]
we rewrite \eqref{eq:abstract-tangential-raw} as
\begin{align*}
        0
        &=
        u(x)-u(z)
        +
        \frac d2(Du(x)+Du(z))\cdot\nu
        +
        \frac{\sqrt3 d}{2}(Du(x)-Du(z))\cdot\nu_i
        \\
        &\qquad
        +
        O\bigl(Md\,\omega(d)\bigr).
\end{align*}
On the other hand, the two expansions at $x$ and $z$ evaluated at each other
also imply
\[
        u(x)-u(z)
        +
        \frac d2(Du(x)+Du(z))\cdot\nu
        =
        O\bigl(Md\,\omega(d)\bigr).
\]
Therefore,
\begin{equation}\label{eq:abstract-tangential-component}
        |(Du(x)-Du(z))\cdot\nu_i|
        \le
        C M\omega(d)
\end{equation}
for every $i=2,\ldots,n$.

Combining \eqref{eq:abstract-normal-component} and
\eqref{eq:abstract-tangential-component}, and using that
$\{\nu,\nu_2,\ldots,\nu_n\}$ is an orthonormal basis, we obtain
\[
        |Du(x)-Du(z)|
        \le
        C_n M\omega(|x-z|)
\]
whenever $0<|x-z|<R_0/4$.  This proves
\eqref{eq:small-gradient-modulus}.

It remains only to handle pairs with $|x-z|\ge R_0/4$.  We first record a
uniform bound for the gradient.  Fix $y\in\Omega'$ and $e\in\mathbb S^{n-1}$.
Taking $\rho=R_0$ in \eqref{eq:abstract-pointwise-expansion} and evaluating
the expansion at $y+R_0e$, we get
\[
        R_0|Du(y)\cdot e|
        \le
        |u(y+R_0e)-u(y)|
        +
        MR_0\omega(R_0).
\]
Hence
\[
        \|Du\|_{L^\infty(\Omega')}
        \le
        C_n\left(
        \frac{\|u\|_{L^\infty(\Omega)}}{R_0}
        +
        M\omega(R_0)
        \right).
\]
If now $|x-z|\ge R_0/4$, then, since $\omega$ is nondecreasing,
\[
        \omega(|x-z|)
        \ge
        \omega(R_0/4).
\]
Therefore
\[
        \frac{|Du(x)-Du(z)|}{\omega(|x-z|)}
        \le
        \frac{C_n}{\omega(R_0/4)}
        \left(
        \frac{\|u\|_{L^\infty(\Omega)}}{R_0}
        +
        M\omega(R_0)
        \right).
\]
Together with the small-distance estimate, this proves
\eqref{eq:global-gradient-modulus}.
\end{proof}

As a corollary, we have the following result.

\begin{corollary}\label{cor:interior-log-holder}
Let \(Q\) and \(m\) satisfy \eqref{eq:Q-choice}--\eqref{eq:m-choice}.
Let \(F\in\mathfrak{F}_{\lambda,\Lambda}\), and let \(u\) be an
\(L^p\)-viscosity solution of
\[
        F(D^2u)=f
        \quad\text{in } \Omega,
\]
with \(f\in L^p(\Omega)\). Then \(u\) is differentiable in \(\Omega\).
Moreover, for every \(\Omega'\Subset\Omega\), set
\[
        \rho_{\Omega'}:=
        \min\left\{
        1,\frac14\dist(\Omega',\partial\Omega)
        \right\},
        \qquad
        D_{\Omega'}:=\diam(\Omega')+2\rho_{\Omega'} ,
\]
with the convention that \(\dist(\Omega',\partial\Omega)=+\infty\) if
\(\Omega=\mathbb R^n\).  Then there exists
\[
        C
        =
        C\left(
        n,p,\lambda,\Lambda,C_H,\alpha_H,Q,m,\rho_{\Omega'}^{-1}
        \right)
\]
such that
\begin{equation}\label{eq:gradient-log-holder}
        \abs{Du(x)-Du(z)}
        \le
        C\left[
        \norm{u}_{L^\infty(\Omega)}
        +
        \norm{f}_{L^p(\Omega)}
        \right]
        \abs{x-z}^{\alpha_H}
        \left(
        1+\log\frac{D_{\Omega'}}{\abs{x-z}}
        \right)^m
\end{equation}
for every \(x,z\in\Omega'\), \(x\neq z\).
\end{corollary}

\begin{proof}
Let
\[
        M:=\norm{u}_{L^\infty(\Omega)}+\norm{f}_{L^p(\Omega)}.
\]
By the definition of \(\rho_{\Omega'}\), we have
\[
        B_{2\rho_{\Omega'}}(y)\Subset\Omega
        \qquad\text{for every } y\in\Omega'.
\]
Applying Theorem~\ref{thm:pointwise-log} in each ball
\(B_{2\rho_{\Omega'}}(y)\), \(y\in\Omega'\), we obtain a constant
\[
        C
        =
        C\left(
        n,p,\lambda,\Lambda,C_H,\alpha_H,Q,m,\rho_{\Omega'}^{-1}
        \right)
\]
such that, for every \(y\in\Omega'\) and every
\(0<\varrho\le \rho_{\Omega'}\),
\[
        \sup_{B_\varrho(y)}
        \abs{u(x)-u(y)-Du(y)\cdot(x-y)}
        \le
        CM\varrho^{1+\alpha_H}
        \left(
        1+\log\frac{D_{\Omega'}}{\varrho}
        \right)^m .
\]
Indeed, the quantity \(\Theta_{2\rho_{\Omega'}}(y)\) appearing in
Theorem~\ref{thm:pointwise-log} is bounded by \(CM\), uniformly for
\(y\in\Omega'\), after allowing \(C\) to depend on
\(\rho_{\Omega'}^{-1}\).  Moreover,
\[
        2\rho_{\Omega'}\le D_{\Omega'},
\]
so the logarithmic factor with numerator \(2\rho_{\Omega'}\) is controlled by
the one displayed above.

Let
\[
        \omega_0(\varrho)
        :=
        \varrho^{\alpha_H}
        \left(
        1+\log\frac{D_{\Omega'}}{\varrho}
        \right)^m .
\]
If necessary, replace \(\omega_0\) by its nondecreasing envelope on
\([0,D_{\Omega'}]\).  This changes the estimates only by a multiplicative
constant depending on \(\alpha_H\) and \(m\).  Applying
Lemma~\ref{lem:pointwise-to-gradient-modulus} gives
\[
        \abs{Du(x)-Du(z)}
        \le
        CM
        \abs{x-z}^{\alpha_H}
        \left(
        1+\log\frac{D_{\Omega'}}{\abs{x-z}}
        \right)^m
\]
for every \(x,z\in\Omega'\), \(x\neq z\), after increasing the constant to
account for the fixed large-distance range.  This is the desired estimate.
\end{proof}

\section{The borderline case \texorpdfstring{$\alpha_H=1-\frac np$}{alphaH = 1 - n/p}}
\label{sec:borderline-equality}

We now turn to the critical threshold case, $\alpha_H=1-\frac np$. At this scale the logarithmic selection used in the strict regime no longer
closes.  Indeed, there is no admissible \(Q>1\) satisfying
\[
        Q\le \frac{2-\frac np}{1+\alpha_H}.
\]
The obstruction is not merely technical: along the logarithmic scales
\(L(r_{k+1})=QL(r_k)\), the perturbative term grows exponentially relative to
any fixed logarithmic normalization. No finite power of the logarithm can compensate for this exponential growth. We overcome this by slowing down the logarithmic scale selection.
\begin{theorem}\label{thm:pointwise-borderline}
Let \(A>0\) be chosen so that
\begin{equation}
\label{eq:borderline-A-choice}
        \mu\coloneqq C_H e^{-A}<1 
        \quad \text{and} \quad
        A>2(1+\alpha_H).
\end{equation}
Let \(F\in\mathfrak{F}_{\lambda,\Lambda}\), and let
\(u\in C(\overline{B_R(x_0)})\) be an \(L^p\)-viscosity solution of
\[
        F(D^2u)=f
        \qquad\text{in }B_R(x_0),
\]
with \(f\in L^p(B_R(x_0))\). Define
\[
        \Theta_R(x_0)
        \coloneqq
        \frac{\Adist{B_R(x_0)}{u}}{R^{1+\alpha_H}}
        +
        \norm{f}_{L^p(B_R(x_0))}.
\]
Then \(u\) is differentiable at \(x_0\), and
\begin{equation}
\label{eq:pointwise-borderline-taylor}
        \abs{
        u(x)-u(x_0)-Du(x_0)\cdot(x-x_0)
        }
        \le
        C\Theta_R(x_0)
        \abs{x-x_0}^{1+\alpha_H}
        \exp\left(
        A\sqrt{1+\log\frac{R}{\abs{x-x_0}}}
        \right)
\end{equation}
for every \(x\in B_R(x_0)\). The constant \(C\) depends only on $n,p,\lambda,\Lambda,C_H,\alpha_H,A .$
\end{theorem}

\begin{proof}
Set
\[    
        F_0\coloneqq \norm{f}_{L^p(B_R(x_0))},
        \quad \text{and} \quad
        L(s)\coloneqq L_R(s) = 1+\log\frac Rs.
\]
We first prove the excess estimate
\begin{equation}\label{eq:borderline-excess-claim}
        \Adist{B_s(x_0)}{u}
        \le
        C\Theta_R(x_0)
        s^{1+\alpha_H}
        \exp\left(A\sqrt{L(s)}\right)
        \qquad\text{for every }0<s\le R .
\end{equation}
Define a decreasing sequence \((r_k)_{k\ge0}\) by
\[
        L(r_k)=(k+1)^2,
        \qquad\text{equivalently}\qquad
        r_k=R\exp\{1-(k+1)^2\}.
\]
Thus \(r_0=R\), \(r_k\downarrow0\), and
\[
        L(r_{k+1})-L(r_k)
        =
        2k+3.
\]
Set
\[
        E_k\coloneqq \Adist{B_{r_k}(x_0)}{u},
        \qquad
        B_k
        \coloneqq
        \frac{E_k}
        {r_k^{1+\alpha_H}\exp(A(k+1))}.
\]
Applying \Cref{lem:two-scale} with \(s=r_{k+1}\) and \(r=r_k\), and using
\[
        2-\frac np=1+\alpha_H,
\]
we obtain
\[
        E_{k+1}
        \le
        C_H\left(\frac{r_{k+1}}{r_k}\right)^{1+\alpha_H}E_k
        +
        C r_k^{1+\alpha_H}F_0 .
\]
Dividing by \(r_{k+1}^{1+\alpha_H}\exp(A(k+2))\), the first term becomes
\[
        C_H e^{-A}B_k
        =
        \mu B_k,
\]
while the forcing term 
\[
\begin{aligned}
        C F_0
        \frac{r_k^{1+\alpha_H}}
        {r_{k+1}^{1+\alpha_H}\exp(A(k+2))}
        &=
        C F_0
        \exp\left\{
        (1+\alpha_H)\bigl(L(r_{k+1})-L(r_k)\bigr)
        -A(k+2)
        \right\}                                      \\
        &=
        C F_0
        \exp\left\{
        (1+\alpha_H)(2k+3)-A(k+2)
        \right\}.
\end{aligned}
\]
Because \(A>2(1+\alpha_H)\), the last exponential factor is uniformly bounded
in \(k\). Hence
\[
        B_{k+1}
        \le
        \mu B_k+CF_0,
        \qquad
        \mu<1.
\]
Iterating,
\[
        B_k
        \le
        C(B_0+F_0)
        \le
        C\Theta_R(x_0)
        \qquad\text{for every }k\ge0.
\]
Therefore
\begin{equation}
\label{eq:borderline-discrete-excess}
        E_k
        \le
        C\Theta_R(x_0)
        r_k^{1+\alpha_H}
        \exp\left(A\sqrt{L(r_k)}\right).
\end{equation}

We now pass from the discrete scales to an arbitrary scale. Let
\(0<s\le R\), and choose \(k\ge0\) such that
\[
        r_{k+1}<s\le r_k .
\]
Applying \Cref{lem:two-scale} with \(r=r_k\), we get
\[
        \Adist{B_s(x_0)}{u}
        \le
        C_H\left(\frac{s}{r_k}\right)^{1+\alpha_H}E_k
        +
        C r_k^{1+\alpha_H}F_0 .
\]
Using \eqref{eq:borderline-discrete-excess} and \(L(s)\ge L(r_k)\), the first
term is bounded by
\[
        C\Theta_R(x_0)
        s^{1+\alpha_H}
        \exp\left(A\sqrt{L(s)}\right).
\]

It remains to control the forcing term. Write \(t=L(s)\). Since
\(r_{k+1}<s\le r_k\), we have
\[
        (k+1)^2=L(r_k)\le t<L(r_{k+1})=(k+2)^2.
\]
Therefore
\[
\begin{aligned}
        \frac{r_k^{1+\alpha_H}}
        {s^{1+\alpha_H}\exp(A\sqrt{L(s)})}
        &=
        \exp\left\{
        (1+\alpha_H)(t-(k+1)^2)-A\sqrt t
        \right\}.
\end{aligned}
\]
Since \(t<(k+2)^2\), we have
\[
        t-(k+1)^2
        \le
        (k+2)^2-(k+1)^2
        =
        2k+3
        \le
        2\sqrt t+1.
\]
Hence
\[
        (1+\alpha_H)(t-(k+1)^2)-A\sqrt t
        \le
        \bigl(2(1+\alpha_H)-A\bigr)\sqrt t+(1+\alpha_H).
\]
Because \(A>2(1+\alpha_H)\), this is bounded above uniformly in \(t\). Thus
\[
        r_k^{1+\alpha_H}F_0
        \le
        C\Theta_R(x_0)
        s^{1+\alpha_H}
        \exp\left(A\sqrt{L(s)}\right),
\]
and \eqref{eq:borderline-excess-claim} follows. From here, \eqref{eq:pointwise-borderline-taylor} follows as in the proof of \Cref{thm:pointwise-log}.
\end{proof}

\begin{remark}[Critical quantitative input]
The same observation applies at the critical threshold.  Suppose that, for a
given class of operators, the homogeneous equation satisfies
\[
        \mathrm{dist}_{B_\rho}(h,\mathcal A)
        \le
        C_*\rho^{1+\alpha_*}\|h\|_{L^\infty(B_1)}.
\]
If $\alpha_*=1-n/p$, then the proof of Theorem~\ref{thm:pointwise-borderline}
gives the endpoint estimate
\[
        |u(x)-u(x_0)-Du(x_0)\cdot(x-x_0)|
        \le
        C\Theta_R(x_0)|x-x_0|^{1+\alpha_*}
        \exp\left(
        A\sqrt{1+\log\frac{R}{|x-x_0|}}
        \right),
\]
provided $A>0$ is chosen so that
\[
        C_*e^{-A}<1
        \qquad\text{and}\qquad
        A>2(1+\alpha_*).
\]
Thus the critical endpoint modulus is also determined quantitatively by the
homogeneous affine approximation constants.
\end{remark}

\begin{corollary}[Interior gradient modulus at the critical threshold]
\label{cor:interior-critical-gradient-modulus}
Assume that
\[
        \alpha_H=\sigma_p=1-\frac np .
\]
Let \(A>0\) be chosen so that
\[
        C_H e^{-A}<1
        \qquad\text{and}\qquad
        A>2(1+\alpha_H).
\]
Let \(F\in\mathfrak F_{\lambda,\Lambda}\), and let \(u\) be an
\(L^p\)-viscosity solution of
\[
        F(D^2u)=f
        \quad\text{in }\Omega,
\]
with \(f\in L^p(\Omega)\).  Then \(u\) is differentiable in \(\Omega\).
Moreover, for every \(\Omega'\Subset\Omega\), set
\[
        \rho_{\Omega'}:=
        \min\left\{
        1,\frac14\dist(\Omega',\partial\Omega)
        \right\},
        \qquad
        D_{\Omega'}:=\diam(\Omega')+2\rho_{\Omega'} ,
\]
with the convention that \(\dist(\Omega',\partial\Omega)=+\infty\) if
\(\Omega=\mathbb R^n\).  Then there exists
\[
        C
        =
        C\left(
        n,p,\lambda,\Lambda,C_H,\alpha_H,A,
        \rho_{\Omega'}^{-1}
        \right)
\]
such that
\begin{equation}\label{eq:gradient-critical-modulus}
        \abs{Du(x)-Du(z)}
        \le
        C\left[
        \norm{u}_{L^\infty(\Omega)}
        +
        \norm{f}_{L^p(\Omega)}
        \right]
        \abs{x-z}^{\alpha_H}
        \exp\left(
        A\sqrt{
        1+\log\frac{D_{\Omega'}}{\abs{x-z}}
        }
        \right)
\end{equation}
for every \(x,z\in\Omega'\), \(x\neq z\).
\end{corollary}

\begin{proof}
The proof follows the same argument as that of
Corollary~\ref{cor:interior-log-holder}.  One applies
Theorem~\ref{thm:pointwise-borderline} in each ball
\(B_{2\rho_{\Omega'}}(y)\Subset\Omega\), \(y\in\Omega'\).  Since
\(2\rho_{\Omega'}\le D_{\Omega'}\), the resulting pointwise expansion is
uniformly controlled by the gauge
\[
        \omega_0(\rho)
        =
        \rho^{\alpha_H}
        \exp\left(
        A\sqrt{1+\log\frac{D_{\Omega'}}{\rho}}
        \right).
\]
Applying Lemma~\ref{lem:pointwise-to-gradient-modulus}, if necessary to the
nondecreasing envelope of \(\omega_0\), gives
\eqref{eq:gradient-critical-modulus}.  We omit the remaining details.
\end{proof}

\section{An abstract scale-selection principle for endpoint Campanato recurrences}\label{sec:abstract}

We now extract the iterative core of the preceding arguments.  The point is
that no specific feature of the equation remains once the two-scale excess
inequality has been proved.  The endpoint obstruction is purely Campanato-theoretic:
the homogeneous decay has the correct exponent, but its constant prevents
contraction on fixed scales.  The scale-selection principle below compensates
for this constant by modifying the radii rather than lowering the exponent.

Consider $R$ and $\beta$ to be positive, and let
\[
        E \colon (0,R]\to[0,\infty)
\]
be a nonnegative function. We think of $E(r)$ as an excess measuring the
distance, at scale $r$, from a class of comparison objects.  The model
recurrence is
\begin{equation}\label{eq:abstract-recurrence}
        E(s)
        \le
        C_0\left(\frac{s}{r}\right)^\beta E(r)
        +
        A r^\eta,
        \qquad 0<s<r\le R,
\end{equation}
where $C_0\ge1$, $A\ge0$, and $\eta\ge\beta$.

The difficulty at the endpoint exponent $\beta$ is the possible lack of
contraction in the homogeneous term.  On fixed geometric scales, the coefficient $C_0(s/r)^\beta$ is made contractive only by choosing $s/r$ sufficiently small. However, this fixed choice may be incompatible with the scaling of the perturbative term at the limiting exponent.  The following theorem shows that the obstruction can be compensated by selecting the radii according to a logarithmic law.

\begin{theorem}[Endpoint Campanato principle: supercritical perturbations]
\label{thm:abstract-log-campanato}
Assume that $E$ satisfies \eqref{eq:abstract-recurrence} with
\[
        \eta > \beta.
\]
Let
\begin{equation}\label{eq:abstract-Q}
        1<Q\le \frac{\eta}{\beta},
\end{equation}
and choose any real number \(m>0\) such that
\begin{equation}\label{eq:abstract-m}
        \mu \coloneqq C_0Q^{-m}<1.
\end{equation}
Then there exists a constant $C$, depending only on
$C_0,\beta,\eta,Q,m$, such that
\begin{equation}\label{eq:abstract-log-conclusion}
        E(s)
        \le
        C\left(
        \frac{E(R)}{R^\beta}
        +
        A R^{\eta-\beta}
        \right)
        s^\beta
        \left(1+\log\frac Rs\right)^m
\end{equation}
for every $0<s\le R$.
\end{theorem}

\begin{proof}
Set
\[
        L(s) \coloneqq 1+\log\frac Rs .
\]
Define a decreasing sequence $(r_k)_{k\ge0}$ by
\[
        L(r_k)=Q^k,
        \qquad
        \text{equivalently}
        \qquad
        r_k=R\exp(1-Q^k).
\]
Then $r_0=R$, $r_k\downarrow0$, and
\[
        L(r_{k+1})=Q L(r_k).
\]
Set
\[
        E_k \coloneqq E(r_k),
        \qquad
        Y_k \coloneqq 
        \frac{E_k}{r_k^\beta L(r_k)^m}.
\]
Applying \eqref{eq:abstract-recurrence} with $s=r_{k+1}$ and $r=r_k$ gives
\[
        E_{k+1}
        \le
        C_0\left(\frac{r_{k+1}}{r_k}\right)^\beta E_k
        +
        A r_k^\eta .
\]
Dividing by $r_{k+1}^\beta L(r_{k+1})^m$, we obtain
\[
        Y_{k+1}
        \le
        C_0
        \left(\frac{L(r_k)}{L(r_{k+1})}\right)^m
        Y_k
        +
        A\frac{r_k^\eta}{r_{k+1}^\beta L(r_{k+1})^m}.
\]
The first coefficient is $C_0Q^{-m}=\mu<1$.

For the second term, write $L_k \coloneqq L(r_k)=Q^k$.  Since
\[
        r_k=R e^{1-L_k},
        \qquad
        r_{k+1}=R e^{1-Q L_k},
\]
we have
\begin{align*}
        \frac{r_k^\eta}{r_{k+1}^\beta L(r_{k+1})^m}
        &=
        R^{\eta-\beta}
        \frac{
        \exp\left\{\eta(1-L_k)-\beta(1-QL_k)\right\}
        }
        {(QL_k)^m}
        \\
        &=
        R^{\eta-\beta}
        \frac{
        \exp\left\{\eta-\beta-(\eta-\beta Q)L_k\right\}
        }
        {(QL_k)^m}.
\end{align*}
By \eqref{eq:abstract-Q}, $\eta-\beta Q\ge0$. Therefore
\[
        \frac{r_k^\eta}{r_{k+1}^\beta L(r_{k+1})^m}
        \le
        C R^{\eta-\beta}.
\]
Hence
\[
        Y_{k+1}
        \le
        \mu Y_k + C A R^{\eta-\beta}.
\]
Iterating,
\[
        Y_k
        \le
        C\left(
        Y_0 + A R^{\eta-\beta}
        \right)
        =
        C\left(
        \frac{E(R)}{R^\beta}
        +
        A R^{\eta-\beta}
        \right).
\]
Thus
\begin{equation}\label{eq:abstract-discrete}
        E(r_k)
        \le
        C\left(
        \frac{E(R)}{R^\beta}
        +
        A R^{\eta-\beta}
        \right)
        r_k^\beta L(r_k)^m .
\end{equation}

It remains to pass from the discrete sequence to an arbitrary scale.  Let
$0<s\le R$, and choose $k\ge0$ such that
\[
        r_{k+1}<s\le r_k.
\]
Using \eqref{eq:abstract-recurrence} once more, with $r=r_k$, gives
\[
        E(s)
        \le
        C_0\left(\frac{s}{r_k}\right)^\beta E(r_k)
        +
        A r_k^\eta .
\]
The first term is bounded by
\[
        C\left(
        \frac{E(R)}{R^\beta}
        +
        A R^{\eta-\beta}
        \right)
        s^\beta L(s)^m,
\]
because $L(s)\ge L(r_k)$.

For the perturbative term, since $s>r_{k+1}$,
\[
        \frac{r_k^\eta}{s^\beta L(s)^m}
        \le
        \frac{r_k^\eta}{r_{k+1}^\beta}
        =
        R^{\eta-\beta}
        \exp\left\{
        \eta-\beta-(\eta-\beta Q)L_k
        \right\}
        \le
        C R^{\eta-\beta}.
\]
Therefore
\[
        A r_k^\eta
        \le
        C A R^{\eta-\beta}s^\beta L(s)^m,
\]
and \eqref{eq:abstract-log-conclusion} follows.
\end{proof}

The borderline case $\eta=\beta$ requires a slower selection of radii.  No
finite power of the logarithm can compensate the perturbative accumulation
generated by the recurrence.  The appropriate replacement is a sub-power
correction.

\begin{theorem}[Endpoint Campanato principle: critical perturbations]
\label{thm:abstract-critical-campanato}
Assume that $E$ satisfies
\begin{equation}\label{eq:abstract-critical-recurrence}
        E(s)
        \le
        C_0\left(\frac{s}{r}\right)^\beta E(r)
        +
        A r^\beta,
        \qquad 0<s<r\le R.
\end{equation}
Let $B>0$ be chosen so that
\begin{equation}\label{eq:abstract-critical-B}
        C_0e^{-B}<1
        \qquad\text{and}\qquad
        B>2\beta .
\end{equation}
Then there exists a constant $C$, depending only on $C_0,\beta,B$, such that
\begin{equation}\label{eq:abstract-critical-conclusion}
        E(s)
        \le
        C\left(
        \frac{E(R)}{R^\beta}
        +
        A
        \right)
        s^\beta
        \exp\left(
        B\sqrt{1+\log\frac Rs}
        \right)
\end{equation}
for every $0<s\le R$.
\end{theorem}

\begin{proof}
Again set
\[
        L(s) \coloneqq 1+\log\frac Rs .
\]
Define $(r_k)_{k\ge0}$ by
\[
        L(r_k)=(k+1)^2,
        \qquad
        \text{equivalently}
        \qquad
        r_k=R\exp\{1-(k+1)^2\}.
\]
Then $r_0=R$, $r_k\downarrow0$, and
\[
        L(r_{k+1})-L(r_k)=2k+3.
\]
Set
\[
        E_k\coloneqq E(r_k),
        \qquad
        Y_k \coloneqq 
        \frac{E_k}{r_k^\beta e^{B(k+1)}}.
\]
Applying \eqref{eq:abstract-critical-recurrence} with $s=r_{k+1}$ and $r=r_k$ yields
\[
        E_{k+1}
        \le
        C_0\left(\frac{r_{k+1}}{r_k}\right)^\beta E_k
        +
        A r_k^\beta .
\]
Dividing by $r_{k+1}^\beta e^{B(k+2)}$, the homogeneous term becomes
\[
        C_0e^{-B}Y_k.
\]
The perturbative term becomes
\begin{align*}
        A\frac{r_k^\beta}{r_{k+1}^\beta e^{B(k+2)}}
        &=
        A\exp\left\{
        \beta\big(L(r_{k+1})-L(r_k)\big)-B(k+2)
        \right\}
        \\
        &=
        A\exp\left\{
        \beta(2k+3)-B(k+2)
        \right\}.
\end{align*}
By \eqref{eq:abstract-critical-B}, this expression is uniformly bounded by
$CA$. Hence
\[
        Y_{k+1}
        \le
        \mu Y_k + CA,
        \qquad
        \mu \coloneqq C_0e^{-B}<1.
\]
Iterating,
\[
        Y_k
        \le
        C(Y_0+A)
        =
        C\left(
        \frac{E(R)}{R^\beta e^B}
        +
        A
        \right).
\]
Therefore
\begin{equation}\label{eq:abstract-critical-discrete}
        E(r_k)
        \le
        C\left(
        \frac{E(R)}{R^\beta}
        +
        A
        \right)
        r_k^\beta
        \exp\left(
        B\sqrt{L(r_k)}
        \right).
\end{equation}

Let now $0<s\le R$, and choose $k\ge0$ such that
\[
        r_{k+1}<s\le r_k.
\]
By \eqref{eq:abstract-critical-recurrence},
\[
        E(s)
        \le
        C_0\left(\frac{s}{r_k}\right)^\beta E(r_k)
        +
        A r_k^\beta .
\]
The first term is bounded by the desired right-hand side, since
$L(s)\ge L(r_k)$.

For the second term, write $t=L(s)$.  Since
\[
        (k+1)^2=L(r_k)\le t<L(r_{k+1})=(k+2)^2,
\]
we have
\[
        t-(k+1)^2
        \le
        (k+2)^2-(k+1)^2
        =
        2k+3
        \le
        2\sqrt t+1.
\]
Hence
\begin{align*}
        \frac{r_k^\beta}{s^\beta e^{B\sqrt{L(s)}}}
        &=
        \exp\left\{
        \beta\big(t-(k+1)^2\big)-B\sqrt t
        \right\}
        \\
        &\le
        \exp\left\{
        (2\beta-B)\sqrt t+\beta
        \right\}
        \le
        C,
\end{align*}
because $B>2\beta$. Therefore
\[
        A r_k^\beta
        \le
        C A s^\beta e^{B\sqrt{L(s)}}.
\]
This proves \eqref{eq:abstract-critical-conclusion}.
\end{proof}

\begin{remark}
Theorems~\ref{thm:abstract-log-campanato} and
\ref{thm:abstract-critical-campanato} separate the analytic input from the
iteration. In applications, the analytic part consists in proving a two-scale excess inequality of the form \eqref{eq:abstract-recurrence}.  Once this is available, the endpoint decay follows from the scale-selection mechanism alone. This is the reason the method is not tied to the particular elliptic equation considered here.
\end{remark}

\section{Endpoint examples and the question of sharpness}
\label{sec:optimality}

We conclude with examples that clarify the role of logarithmic corrections at
endpoint scales.  The purpose of this section is not to prove the optimality
of the moduli in \Cref{thm:pointwise-log,thm:pointwise-borderline}.  Rather,
we isolate the scaling mechanisms that make logarithmic losses natural at
critical thresholds, and we explain why the sharpness question in the strict
homogeneous-limited regime remains more delicate.

We begin with a scaling observation.  The estimates proved in this paper are
driven by recursive excess inequalities, and in such arguments the relevant
feature of the right-hand side is its behavior under rescaling.  For
\[
        F(D^2u)=f,
\]
the rescaled source in $B_1$ is
\[
        f_r(y):=r^2f(x_0+ry),
\]
and
\[
        \|f_r\|_{L^p(B_1)}
        =
        r^{2-\frac np}\|f\|_{L^p(B_r(x_0))}.
\]
The same relation holds in weak spaces:
\[
        \|f_r\|_{L^{p,\infty}(B_1)}
        =
        r^{2-\frac np}\|f\|_{L^{p,\infty}(B_r(x_0))}.
\]
Thus the scale governing the perturbation is the exponent
\[
        \sigma_p=1-\frac np,
\]
rather than the particular choice of strong or weak $L^p$ norm.  This viewpoint
is consistent with the borderline theory of Daskalopoulos, Kuusi and Mingione
\cite{DaskalopoulosKuusiMingione2014}, where endpoint estimates are formulated
through scale-sensitive control of the datum.

Classical linear examples already show that logarithmic losses are natural at
critical scales.  For instance, in $\mathbb R^n$, $n\ge2$, the function
\[
        u_0(x)=x_1\log\frac e{|x|},
\]
smoothly truncated away from the origin, is log-Lipschitz but not Lipschitz at
the origin, while
\[
        |\Delta u_0(x)|\le \frac{C}{|x|}
        \qquad\text{and}\qquad
        |x|^{-1}\in L^{n,\infty}(B_{1/2}).
\]
A softer logarithm gives a strong $L^n$ datum: if
\[
        \widetilde L(r):=\log\frac{e^e}{r},
        \qquad
        u_1(x):=x_1\log \widetilde L(|x|),
\]
then $u_1$ is still log-Lipschitz and not Lipschitz, whereas
\[
        |\Delta u_1(x)|
        \le
        \frac{C}{|x|\widetilde L(|x|)}
\]
and
\[
        \int_{B_{1/2}} |\Delta u_1|^n\,dx
        \le
        C\int_0^{1/2}
        \frac{dr}{r\widetilde L(r)^n}
        <\infty.
\]
Similarly, if $P_2$ is a nontrivial harmonic homogeneous polynomial of degree
two, then
\[
        v(x)=P_2(x)\log\frac e{|x|}
\]
has bounded Laplacian near the origin, while $Dv$ has the sharp modulus
$r\log(e/r)$.  These elementary models reflect the logarithmic endpoint
phenomena captured, in the fully nonlinear setting, by the universal moduli of
continuity in \cite{Teixeira2014}.

We now turn to the fully nonlinear obstruction relevant to the present paper.
Here the limiting exponent is not imposed by the source term, but by the
homogeneous equation itself.  The following computation shows how singular
homogeneous profiles interact with logarithmic perturbations at the critical
scale.

\begin{proposition}[Singular-profile test at the homogeneous endpoint]
\label{prop:singular-profile-test}
Assume that $F$ is positively homogeneous and uniformly elliptic.  Let $h$ be
a nontrivial homogeneous solution of
\[
        F(D^2h)=0
        \qquad\text{in } B_1\setminus\{0\},
\]
smooth away from the origin, with homogeneity $1+\alpha$, $0<\alpha<1$:
\[
        h(rx)=r^{1+\alpha}h(x).
\]
Let
\[
        u(x):=h(x)\varphi(|x|),
        \qquad r:=|x|.
\]
Then, near the origin,
\[
        |F(D^2u)|
        \le
        C\left(
        r^\alpha|\varphi'(r)|
        +
        r^{1+\alpha}|\varphi''(r)|
        \right).
\]

Consequently, if
\[
        L(r):=1+\log\frac1r
        \qquad\text{and}\qquad
        \varphi(r)=L(r)^m,
\]
then
\[
        |F(D^2u)|
        \le
        C r^{\alpha-1}L(r)^{m-1}.
\]
In particular, for
\[
        p_\alpha:=\frac{n}{1-\alpha},
        \qquad
        \sigma_{p_\alpha}=\alpha,
\]
the case $m=1$ produces a right-hand side of critical weak size:
\[
        F(D^2u)\in L^{p_\alpha,\infty}_{\mathrm{loc}}.
\]
For $m>1$, the logarithmic factor pushes the source beyond this weak endpoint
scale.

If instead
\[
        \ell(r):=\log L(r),
        \qquad
        \varphi(r)=\ell(r)^m,
\]
then
\[
        |F(D^2u)|
        \le
        C r^{\alpha-1}\frac{\ell(r)^{m-1}}{L(r)},
\]
and hence
\[
        F(D^2u)\in L^{p_\alpha}_{\mathrm{loc}}.
\]
Thus logarithmic deterioration of the homogeneous $C^{1,\alpha}$ behavior is
compatible with the critical scale $\sigma_p=\alpha$.
\end{proposition}

\begin{proof}
Since $h$ is homogeneous of degree $1+\alpha$, we have, near the origin,
\[
        |h(x)|\le Cr^{1+\alpha},
        \qquad
        |\nabla h(x)|\le Cr^\alpha,
        \qquad
        |D^2h(x)|\le Cr^{\alpha-1}.
\]
Moreover,
\[
        D^2u
        =
        \varphi D^2h
        +
        \nabla h\otimes\nabla\varphi
        +
        \nabla\varphi\otimes\nabla h
        +
        hD^2\varphi .
\]
Since $F$ is positively homogeneous and $F(D^2h)=0$,
\[
        F(\varphi D^2h)=\varphi F(D^2h)=0.
\]
Uniform ellipticity therefore gives
\[
        |F(D^2u)|
        \le
        C\left(
        r^\alpha|\varphi'(r)|
        +
        r^{1+\alpha}|\varphi''(r)|
        \right).
\]

For $\varphi(r)=L(r)^m$, one has
\[
        |\varphi'(r)|\le Cr^{-1}L(r)^{m-1},
        \qquad
        |\varphi''(r)|\le Cr^{-2}L(r)^{m-1},
\]
and therefore
\[
        |F(D^2u)|
        \le
        Cr^{\alpha-1}L(r)^{m-1}.
\]
At $p_\alpha=n/(1-\alpha)$,
\[
        r^{\alpha-1}=r^{-\frac n{p_\alpha}},
\]
which is precisely the weak-$L^{p_\alpha}$ threshold.  Thus
$r^{\alpha-1}\in L^{p_\alpha,\infty}_{\mathrm{loc}}$ but
$r^{\alpha-1}\notin L^{p_\alpha}_{\mathrm{loc}}$.  This proves the weak
critical assertion for $m=1$, while the additional factor $L(r)^{m-1}$ with
$m>1$ moves the source beyond the weak endpoint scale.

For $\varphi(r)=\ell(r)^m$, where $\ell(r)=\log L(r)$, we have
\[
        |\varphi'(r)|
        \le
        C\frac{\ell(r)^{m-1}}{rL(r)},
        \qquad
        |\varphi''(r)|
        \le
        C\frac{\ell(r)^{m-1}}{r^2L(r)}.
\]
Hence
\[
        |F(D^2u)|
        \le
        Cr^{\alpha-1}\frac{\ell(r)^{m-1}}{L(r)}.
\]
Consequently,
\[
        \int_{B_{1/2}} |F(D^2u)|^{p_\alpha}\,dx
        \le
        C\int_0^{1/2}
        \frac{\ell(r)^{p_\alpha(m-1)}}{rL(r)^{p_\alpha}}\,dr
        <\infty,
\]
because $p_\alpha>1$.  The proposition follows.
\end{proof}

Singular and nonclassical homogeneous profiles of this nature are naturally
connected with the program of Nadirashvili and Vl\u{a}du\c{t} on singular
solutions and homogeneous profiles for fully nonlinear uniformly elliptic
equations \cite{NadirashviliVladut2007,NadirashviliVladut2008,
NadirashviliVladut2011,NadirashviliVladut2013a,
NadirashviliVladut2013b}. The proposition shows that logarithmic losses are
genuinely visible at the critical homogeneous scale.  At
\[
        \sigma_p=\alpha,
\]
multiplying a homogeneous singular profile by logarithmic factors generates
right-hand sides in critical Lorentz or strong $L^p$ spaces, depending on the
strength of the logarithm.

The strict homogeneous-limited regime is different.  If
\[
        \alpha<\sigma_p,
\]
then the same singular-profile test cannot produce a logarithmic worsening
while keeping the right-hand side in the desired $L^p$ class.  Indeed, the
condition
\[
        r^{\alpha-1}L(r)^{m-1}\in L^p
\]
requires
\[
        p(1-\alpha)<n,
        \qquad\text{equivalently}\qquad
        \sigma_p<\alpha,
\]
which is precisely the source-limited side of the theory.  To enter the
regime $\sigma_p>\alpha$, one must introduce a genuine positive power gain.  For
example, if
\[
        \varphi(r)
        =
        1+r^\varepsilon
        \log\log\frac{e^e}{r},
\]
then
\[
        |F(D^2(h\varphi))|
        \le
        C r^{\alpha+\varepsilon-1}
        \log\log\frac{e^e}{r},
\]
which belongs to $L^p$ provided
\[
        \sigma_p<\alpha+\varepsilon.
\]
However,
\[
        h(x)\varphi(|x|)-h(x)
        =
        O\left(
        r^{1+\alpha+\varepsilon}
        \log\log\frac{e^e}{r}
        \right),
\]
which is higher order than the homogeneous profile and therefore does not
worsen the leading $C^{1,\alpha}$ behavior.

This is the main obstruction to proving sharpness of the logarithmic defect in
\Cref{thm:pointwise-log} by means of homogeneous singular profiles.  The
critical examples above show that logarithmic losses are natural, and in some
endpoint regimes sharp.  They do not decide, however, whether the
specific logarithmic defect in the strict homogeneous-limited regime
$\alpha_H<\sigma_p$ can be removed, improved, or shown to be sharp.  This
remains an open endpoint question.  Resolving it would likely require a
genuinely inhomogeneous construction, or a finer use of the
Nadirashvili--Vl\u{a}du\c{t} singular-profile program.  The contribution of
the present work is to isolate this obstruction and to provide a quantitative
scale-selection method that reaches the homogeneous differentiability scale up
to an explicit logarithmic loss.

\section*{Acknowledgments}
This publication is based upon work supported by King Abdullah University of
Science and Technology (KAUST). E.V.T. also acknowledges support from the Grayce B. Kerr Chair funds at Oklahoma State University.

\section*{Declarations}

\noindent\textbf{Data availability.}
Data sharing is not applicable to this article, as no datasets were generated
or analyzed during the current study.

\medskip

\noindent\textbf{Competing interests.}
The authors declare that they have no relevant financial or non-financial
interests to disclose.

\medskip

\noindent\textbf{Author contributions.}
Both authors contributed to the conception, development, writing, and revision
of the manuscript.  Both authors read and approved the final manuscript.

\bibliographystyle{amsplain}

\end{document}